\documentclass[a4paper,11pt,fullpage]{article}
\usepackage[english]{babel}

\usepackage[left=1in,right=1in,top=1in,bottom=1in]{geometry}

\usepackage{amsmath}
\usepackage{amsthm}
\usepackage{amsfonts}
\usepackage{amssymb}
\usepackage{amsxtra}
\usepackage{latexsym}
\usepackage{ifthen}
\usepackage{cite}
\usepackage{mathtools}
\usepackage{esint}

\usepackage{epic}
\usepackage{eepic}

\DeclareMathOperator*{\esssup}{ess\,sup}

\DeclareMathOperator*{\essinf}{ess\,inf}

\numberwithin{equation}{section}
\newtheorem{theorem}{Theorem}[section]
\newtheorem{lemma}{Lemma}[section]

\newtheorem{definition}{Definition}[section]
\newtheorem{corollary}{Corollary}[section]

\def\XXint#1#2#3{{\setbox0=\hbox{$#1{#2#3}{\int}$}
     \vcenter{\hbox{$#2#3$}}\kern-.5\wd0}}

\begin{document}

\title{Harnack's inequality for quasilinear elliptic equations with generalized
Orlicz growth
 %Local continuity of solutions to
%singular-degenerate parabolic equations with nonstandard growth
%\thanks{Dedicated to the 80th anniversary of Professor Igor V. Skrypnik.}
}

\author{Maria A. Shan, Igor I. Skrypnik, Mykhailo V. Voitovych %\thanks{The Division of Applied Problems
 %in Contemporary Analysis, Institute of Mathematics of NASU, Kiev, Ukraine}
 }
%\email{iskrypnik@iamm.donbass.com}

 \maketitle

  \begin{abstract}
We prove Harnack's inequality for bounded weak solutions to quasilinear second order elliptic
equations with generalized Orlicz growth conditions.
Our approach covers new cases of variable exponent and $(p, q)$ growth conditions.

\textbf{Keywords:}
Elliptic equations, nonstandard growth, non-log condition, bounded solutions,
Harnack inequality.

\textbf{MSC (2010)}: 35B65, 35D30, 35J60,  49N60.

\end{abstract}

\pagestyle{myheadings} \thispagestyle{plain} \markboth{Maria O.
Shan}{Harnack's inequality for quasilinear elliptic equations with generalized
Orlicz growth}

\section{Introduction and main results}\label{Introduction}

In this paper we are concerned with quasilinear elliptic equations of the form
\begin{equation}\label{gellequation}
{\rm div}\bigg( g(x,|\nabla u|)\frac{\nabla u}{|\nabla u|} \bigg)=0,
\quad x\in \Omega,
\end{equation}
where $\Omega$ is a bounded domain in $\mathbb{R}^{n}$, $n\geqslant2$.

Throughout the paper we suppose that the function
$g(x, {\rm v}): \Omega\times \mathbb{R}_{+}\rightarrow \mathbb{R}_{+}$,
$\mathbb{R}_{+}:=[0,+\infty)$,
satisfies the following assumptions:
\begin{itemize}
\item[(${\rm g}$)]
$g(\cdot, {\rm v})\in L^{1}(\Omega)$ for all ${\rm v}\in \mathbb{R}_{+}$,
$g(x, \cdot)$ is continuous and non-decreasing for almost all $x\in\Omega$,
$\lim\limits_{{\rm v}\rightarrow +0}g(x, {\rm v})=0$ and
$\lim\limits_{{\rm v}\rightarrow +\infty}g(x, {\rm v})=+\infty$;
%$$
%\lim\limits_{{\rm v}\rightarrow +0}g(x, {\rm v})=0 \ \ \text{and} \ \
%\lim\limits_{{\rm v}\rightarrow +\infty}g(x, {\rm v})=+\infty.
%$$
\end{itemize}
%We also assume that
\begin{itemize}
\item[(${\rm g}_{1}$)]
there exist $c_{1}>0$, $q>1$ and $b_{0}\geqslant0$ such that
\begin{equation}\label{gqineq}
\frac{g(x, {\rm w})}{g(x, {\rm v})}\leqslant
c_{1} \left( \frac{{\rm w}}{{\rm v}} \right)^{q-1},
\end{equation}
for all $x\in \Omega$ and for all ${\rm w}\geqslant{\rm v}\geqslant b_{0}$;
\end{itemize}
\begin{itemize}
\item[(${\rm g}_{2}$)]
there exists $p>1$ such that
\begin{equation}\label{gpineq}
\frac{g(x, {\rm w})}{g(x, {\rm v})}\geqslant
\left( \frac{{\rm w}}{{\rm v}} \right)^{p-1},
\end{equation}
for all $x\in \Omega$ and for all ${\rm w}\geqslant{\rm v}>0$;
\end{itemize}
\begin{itemize}
\item[(${\rm g}_{3}$)]
for any $K>0$ and for any ball $B_{8r}(x_{0})\subset\Omega$ there exists $c_{2}(K)>0$
such that
$$
g(x_{1},  {\rm v}/r)\leqslant c_{2}(K)e^{\lambda(r)}g(x_{2},  {\rm v}/r),
$$
for all $x_{1}$, $x_{2}\in B_{r}(x_{0})$ and for all
$r\leqslant {\rm v}\leqslant K$.
Here $\lambda(r):(0, r_{\ast})\rightarrow \mathbb{R}_{+}$
is a continuous, non-increasing function, satisfying
the conditions described below.
\end{itemize}

In addition, it turns out that the following functions
defined on $\Omega\times \mathbb{R}_{+}$ satisfy assumptions
(${\rm g}_{1}$)--(${\rm g}_{3}$):
\begin{equation}\label{examplfnctg12}
g(x, {\rm v}):= {\rm v}^{\,p(x)-1}+{\rm v}^{\,q(x)-1},
\quad
g(x, {\rm v}):={\rm v}^{\,p(x)-1} \big(1+\ln(1+{\rm v}) \big),
\end{equation}
\begin{equation}\label{examplfnctg34}
g(x, {\rm v}):= {\rm v}^{\,p-1}+a(x){\rm v}^{\,q-1}, \
a(x)\geqslant 0, \quad
g(x, {\rm v}):={\rm v}^{\,p-1} \big(1+b(x)\ln(1+{\rm v}) \big), \
b(x)\geqslant 0,
\end{equation}
where the exponents $p$, $q$, $p(\cdot)$, $q(\cdot)$, and the coefficients $a(\cdot)$
and $b(\cdot)$ satisfy the  following conditions:
$1<p< p(x)\leqslant q(x)< q<+\infty$ for all $x\in \Omega$,
\begin{equation}\label{eqSkr1.3}
\begin{aligned}
&|p(x)-p(y)|+|q(x)-q(y)|\leqslant
\frac{\lambda(|x-y|)}
{\big|\ln |x-y|\big|},
%L \left|\ln \left|\ln \frac{\lambda(|z-y|)}{|z-y|}\right| \right|
\quad x,y\in \Omega, \quad x\neq y,
\\
& \text{the function } \frac{\lambda(r)}{|\ln r|}
\text{ is non-decreasing on } (0,r_{\ast})
\text{ for some } r_{\ast}>0, \quad
\lim\limits_{r\rightarrow 0} \frac{\lambda(r)}{|\ln r|}=0,
\hskip 5 mm
\end{aligned}
\end{equation}
\begin{equation}\label{axconditin}
\begin{aligned}
&|a(x)-a(y)|\leqslant A|x-y|^{\alpha}\, e^{\lambda(|x-y|)},
\quad x,y\in \Omega,
\quad x\neq y, \quad A>0, \quad 0<q-p\leqslant\alpha \leqslant 1,
\\
&\text{the function } r^{\alpha}\,e^{\lambda(r)} \text{ is non-decreasing on } (0,r_{\ast})
\text{ for some } r_{\ast}>0, \quad \lim\limits_{r\rightarrow 0}r^{\alpha}\,e^{\lambda(r)}=0,
\end{aligned}
\end{equation}
\begin{equation}\label{bxconditin}
\begin{aligned}
&|b(x)-b(y)|\leqslant  \frac{B\,e^{\lambda(|x-y|)}}{\big|\ln |x-y| \big|},
\quad x,y\in \Omega, \quad x\neq y, \quad B>0,
\\
&\text{the function } \frac{e^{\lambda(r)}}{|\ln r|}\text{ is non-decreasing on } (0,r_{\ast})
\text{ for some } r_{\ast}>0,
\quad \lim\limits_{r\rightarrow 0} \frac{e^{\lambda(r)}}{|\ln r|}=0. \hskip 5 mm
\end{aligned}
\end{equation}

The study of regularity of minima of functionals with non-standard growth
of $(p,q)$-type has been initiated by Zhikov
\cite{ZhikIzv1983, ZhikIzv1986, ZhikJMathPh94, ZhikJMathPh9798, ZhikKozlOlein94},
Marcellini \cite{Marcellini1989, Marcellini1991} and Lieberman \cite{Lieberman91}
and in the last thirty years, the qualitative theory of second order equations with so-called
''log-condition'' (i.e. if $\lambda(r)\leqslant L<+\infty$) has been actively developed
(see, for instance,
\cite{%AcerbiFuscoJDE94, AcerbiMingioneArchRat01, AcerbiMingioneAnSc01, AcerbiMingioneArchRat02,
%AcerbiMingioneJRAngMath05,
Alhutov97, AlhutovKrash04, AlkhSurnApAn19, AlkhSurnTrPetr19,
BarColMing, BarColMingStPt16, BarColMingCalc.Var.18, %BenedMascoloAbsApplAn04,
BurSkr_PotAn,
%ByunOh17, ByunRyuShin18, ChiadoPiatCoscia,
ColMing218, ColMing15, ColMingJFnctAn16,
%ElMarcMas16, ElMarcMasPuraAppl16, ElMarcMasAdvCalc17, GiandiNap13,
HarHastoZ.Anal19, HarHastLee2018,
HarHastToiv17, Mingione2006, VoitNA19}).
These classes of equations have numerous applications in physics and have been
attracted attention for several decades
(see, e.g.,
\cite{AntDiazShm2002_monogr, Ruzicka2000, Weickert}
and references therein).

The case when conditions \eqref{eqSkr1.3}, \eqref{axconditin}, \eqref{bxconditin} hold
differs substantially from the log-case.
To our knowledge there are few results in this direction. Zhikov \cite{ZhikPOMI04}
obtained a generalization of the logarithmic condition which guarantees the density
of smooth functions in Sobolev space $W^{1,p(x)}(\Omega)$. Particularly,
this result holds if $1<p\leqslant p(x)$ and
$$
|p(x)-p(y)|\leqslant L\,
\frac{ \Big|\ln \big|\ln |x-y|\big| \Big|}{\big|\ln |x-y|\big|},
\quad x,y\in\Omega, \quad x\neq y, \quad L<p/n.
$$
In the case when the variable exponent $p(x)$ satisfies the condition
\begin{equation}\label{lnlnlnZhikcond}
|p(x)-p(x_{0})|\leqslant L\,\frac{\ln\ln\ln|x-x_{0}|^{-1}}{\ln|x-x_{0}|^{-1}},
\ \ 0<L<p/(n+1), \ \ x,x_{0}\in\Omega,  \ \ |x-x_{0}|<1/27,
\end{equation}
Alkhutov and Krasheninnikova \cite{AlhutovKrash08}
proved the continuity of solutions to the $p(x)$-Laplace
equation at the point $x_{0}$, and Surnachev \cite{SurnPrepr2018}
established the Harnack inequality for solutions.
The continuity of solutions to the $p(x)$-Laplace equation up to the boundary  were proved
by Alkhutov and Surnachev \cite{AlkhSurnAlgAn19} under the additional condition
\begin{equation}\label{AlkSurncond}
\int_{0}\exp\Big(-C\exp\big( \beta\lambda(r) \big)  \Big)\,
\frac{dr}{r}=+\infty,
\end{equation}
where $C$ and $\beta$ are some positive constants, depending only upon the data.
We note that the function $\lambda(r):=L\ln\ln\ln r^{-1}$, $r\in(0,e^{-e})$,
$L\beta<1$,
satisfies condition \eqref{AlkSurncond}.

In \cite{SkrVoitarXiv20}, we attempted to systematize and unify
the approach to establish the local regularity of bounded solutions of elliptic
and parabolic equations with non-standard growth.
For this, we have introduced  elliptic and parabolic $\mathcal{B}_{1}$ classes,
which generalize the well-known $\mathfrak{B}_{p}$ classes ($p>1$)
of De\,Giorgi, Ladyzhenskaya, Ural'tseva \cite{LadUr}
and cover their other numerous and scattered analogues
(see references in \cite{SkrVoitarXiv20}).
%In \cite{SkrVoitUMB19, SkrVoitarXiv20} the authors were introduced elliptic and parabolic
%$\mathcal{B}_{1}$ classes, which generalized the well-known classes of De\,Giorgi,
%Ladyzhenskaya and Ural'tseva.
It was proved in \cite{SkrVoitarXiv20} that functions from the
$\mathcal{B}_{1,g,\lambda}(\Omega)$ class are continuous
if conditions (${\rm g}_{1}$), (${\rm g}_{3}$)
and \eqref{AlkSurncond} are fulfilled.
In addition, if condition (${\rm g}_{2}$) is fulfilled, then the
solutions of Eq. \eqref{gellequation} belong to the $\mathcal{B}_{1,g,\lambda}(\Omega)$ class.
%The $\mathcal{B}_{1}$-class approach to equation Eq. \eqref{gellequation} is general enough
%and has a number of advantages.
%It allows us to stay within the framework of the ordinary theory of Sobolev spaces:
%to use the embedding $W^{1,1}(\Omega)\subset L^{\frac{n}{n-1}}(\Omega)$ and the classical
% De\,Giorgi-Poincar\'{e} inequality \cite[Chapter~2, Lemma~3.9]{LadUr} as the main tools.
At the same time, we do not use the specific properties of the generalized Orlicz and
Sobolev-Orlicz spaces, as was done, for example, in the papers of Harjulehto, H\"{a}st\"{o} et al
\cite{HarHastOrlicz, HarHastoZ.Anal19, HarHastLee2018, HarHastToiv17}.
Although it should be noted that in the case when $0\leqslant\lambda(r)<L<+\infty$,
the assumptions (${\rm g}_{1}$), (${\rm g}_{2}$), (${\rm g}_{3}$) are almost equivalent
to the conditions $({\rm aDec})_{q}^{\infty}$,
(${\rm A}1$-n), $({\rm aInc})_{p}$ from their papers.

Returning to our paper \cite{SkrVoitarXiv20},
we note that there are no Harnack-type theorems in it.
Although, such type results were obtained in
\cite{Alhutov97, AlhutovKrash04, AlkhSurnApAn19, AlkhSurnTrPetr19, Ok_AdvNonlAn18}
in the log-case and in \cite{SurnPrepr2018} under condition \eqref{lnlnlnZhikcond}.
%In the log-case Harnack's inequality was proved in
%\cite{Alhutov97, AlhutovKrash04, AlkhSurnApAn19, AlkhSurnTrPetr19}.
Therefore, it is natural to conjecture  that the Harnack inequality holds for
bounded solutions of Eq. \eqref{gellequation}
under the conditions (${\rm g}$), (${\rm g}_{1}$)--(${\rm g}_{3}$).
In this paper, we give a positive answer to this hypothesis.
This also encompasses the classic results of
Moser \cite{Moser1961}, Serrin \cite{Serrin}, Trudinger \cite{TrudingerArch71}
and Di\,Benedetto\,$\&$\,Trudinger \cite{DiBenedettoTrud84} for bounded solutions
in the standard growth case, and of course,
we use some of the ideas of Moser and Trudinger in our proofs.

Before formulating the main results, let us remind the reader the definition of a weak solution
to Eq. \eqref{gellequation}.
Moreover, throughout the article, we use the well-known notation for sets in $\mathbb{R}^{n}$,
spaces of functions and their elements, etc. (see, for instance, \cite{LadUr}).
In particular, we will use the notation $\fint\limits_{E}f\,dx:=|E|^{-1}\int\limits_{E}f\,dx$
for any measurable set $E\subset \mathbb{R}^{n}$ with $|E|\neq 0$ and $f\in L^{1}(E)$,
where $|E|$ denotes the $n$-dimensional Lebesgue measure of $E$.

We set
\begin{equation}\label{deffuncgw}
G(x, {\rm v}):=g(x,{\rm v}){\rm v} \ \ \text{for} \ \ {\rm v}\geqslant0
\end{equation}
%\begin{equation}\label{defG}
%\mathcal{G}(x, {\rm w}):=\int\limits_{0}^{{\rm w}}g(x,{\rm v})\,d{\rm v}
%\ \ \text{for} \ \ {\rm w}>0
%\end{equation}
and write $u\in W^{1,G}(\Omega)$ if $u\in W^{1,1}(\Omega)$
and $\int\limits_{\Omega}G(x, |\nabla u|)\,dx<+\infty$;
$u\in W^{1,G}_{{\rm loc}}(\Omega)$ if $u\in W^{1,G}(E)$
for any open set $E$ compactly embedding in $\Omega$.
%for the class of functions $u:\Omega\rightarrow\mathbb{R}$
%which are weakly differentiable in $\Omega$ with ;
We denote by $W^{1,G}_{0}(\Omega)$ the set of all functions
$u\in W^{1,G}(\Omega)$ which have a compact support in $\Omega$.
% $W^{1,G}_{0}(\Omega):=
%\left\{ w\in W^{1,G}(\Omega): w \text{ has a compact support in } \Omega\right\}$.
\begin{definition}
{\rm We say that a function $u:\Omega\rightarrow\mathbb{R}$ is a bounded weak solution (subsolution, supersolution) to Eq. \eqref{gellequation} if
$u\in W^{1,G}_{{\rm loc}}(\Omega)\cap L^{\infty}(\Omega)$ and the integral identity (inequality)
\begin{equation}\label{gelintidentity}
\int\limits_{\Omega}
g(x,|\nabla u|)\,\frac{\nabla u}{|\nabla u|}\,\nabla\varphi\,dx=(\leqslant\,,\geqslant)\,0
\end{equation}
holds for any $\varphi\in W^{1,G}_{0}(\Omega)$
(for subsolutions and supersolutions, we require $\varphi\geqslant 0$).}
\end{definition}

%For any  bounded weak solution (supersolution) $u$ to Eq.~\eqref{gellequation}
%under conditions {\rm (${\rm g}$)},  {\rm (${\rm g}_{1}$)}--{\rm (${\rm g}_{3}$)},
We refer to the parameters
$M:=\esssup\limits_{\Omega} |u|$, $n$, $p$, $q$, $c_{1}$, $c_{2}(M)$
as our structural data,
and we write $\gamma$ if it can be quantitatively determined a priori only in terms
of the above quantities.
The generic constant $\gamma$ may vary from line to line.
%even within the same line.

Our main result of this paper reads as follows.
\begin{theorem}[weak Harnack inequality]\label{thweakHarnck}
Fix a point $x_{0}\in \Omega$ and consider the ball $B_{8\rho}(x_{0})\subset\Omega$.
Let $u$ be a nonnegative bounded weak supersolution to Eq. \eqref{gellequation}
under conditions {\rm (${\rm g}$)}, {\rm (${\rm g}_{1}$)}--{\rm (${\rm g}_{3}$)}.
%with $K=M+2(1+b_{0})$. be fulfilled
Then for any $0<s<n/(n-1)$ there holds:
\begin{multline}\label{weakHrnckineq}
\Bigg(\fint\limits_{B_{5\rho/4}(x_{0})}
g^{s}\left( x_{0},\frac{u+2(1+b_{0})\rho}{\rho} \right)dx\Bigg)^{1/s}
\leqslant\gamma\,\Lambda(\gamma,3n,\rho)\,
%\bigg(1-\frac{n-1}{n}s\bigg)
%^{-\gamma e^{2n\lambda(\rho)}}
g\left( x_{0},\frac{m(\rho)+2(1+b_{0})\rho}{\rho} \right),
\end{multline}
where %positive constant $\gamma$ depends only on the known data,
$m(\rho):=\essinf\limits_{B_{\rho}(x_{0})}u$ and
$\Lambda(c,\beta,\rho):= \exp\Big(c\exp \big( \beta\lambda(\rho)  \big)  \Big)$
for any $c$, $\beta\in \mathbb{R}$ and $\rho\in (0,r_{\ast})$.
%the value of $\Lambda(\gamma,3n,\rho)$
%was defined in Theorem  \ref{sk.th 1.1}.
%$\fint\limits_{E}f\,dx:=|E|^{-1}\int\limits_{E}f\,dx$
%for any measurable set $E\subset \mathbb{R}^{n}$ and $f\in L^{1}(E)$,
%and $|E|$ denotes the $n$-dimensional Lebesgue measure of $E$.
\end{theorem}
\begin{corollary}\label{conttheorem}
Let $u$ be a nonnegative bounded weak solution to Eq. \eqref{gellequation}
under conditions {\rm (${\rm g}$)}, {\rm (${\rm g}_{1}$)}--{\rm (${\rm g}_{3}$)},
and let $\rho_{0}$ be a sufficiently small positive number such that
$B_{8\rho_{0}}(x_{0})\subset\Omega$.
There exist positive numbers $c$, $\beta$ depending only on the data
such that if $\Lambda(c,\beta,r)\leqslant\frac{3}{2}\, \Lambda(c,\beta,2r)$
for all $0<r\leqslant \rho/2<\rho_{0}/2$, and additionally
$$
\int\limits_{0} \Lambda(-c,\beta,r)\,\frac{dr}{r}=+\infty \quad \text{and}
\quad \lim\limits_{r\rightarrow0}\, r\Lambda(c,\beta,r)=0,
$$
then the solution $u$ is continuous at $x_{0}$.
Particularly, the function $\lambda(r)=L\ln\ln\ln r^{-1}$,
$r\in (0,e^{-e})$, satisfies the above conditions if $0<L<1/\beta$.
%\begin{equation}\label{Lambd32cond}
%\Lambda(c,\beta,r)\leqslant
%\frac{3}{2}\, \Lambda(c,\beta,2r)
%\quad \text{for all} \ \ 0<r\leqslant \rho/2<\rho_{0}/2,
%\end{equation}
%then the following inequality holds for all $0<2r<\rho<\rho_{0}$:
%\begin{multline}\label{estellipth2.1}
%{\rm osc}\{u; B_{r}(x_{0})\}:=
%\esssup\limits_{B_{r}(x_{0})}u - \essinf\limits_{B_{r}(x_{0})}u
%\\
%\leqslant
%2M\exp\left( -\gamma\int\limits_{2r}^{\rho}
%\Lambda(-c,\beta,t)\,\frac{dt}{t} \right)
%+\gamma(1+s_{0})\,\rho\,\Lambda(c,\beta,\rho).
%\end{multline}
%If additionally
%\begin{equation}\label{condonLambdabetar}
%\int\limits_{0} \Lambda(-c,\beta,r)\,\frac{dr}{r}=+\infty
%\quad \text{and} \quad
%\lim\limits_{r\rightarrow0}\, r\Lambda(c,\beta,r)=0,
%\end{equation}
%then $u$ is continuous at $x_{0}$.
\end{corollary}
\begin{theorem}[Moser-type sub-estimate of solutions]\label{sk.th 3.1}
%Let all the assumptions  of Theorem \ref{sk.th 1.1} be fulfilled,
Fix a point $x_{0}\in \Omega$ and consider the ball $B_{8\rho}(x_{0})\in\Omega$.
Let conditions {\rm (${\rm g}$)}, {\rm (${\rm g}_{1}$)}--{\rm (${\rm g}_{3}$)}
be fulfilled, and let $u$ be a nonnegative bounded weak solution %with $K=M+2(1+b_{0})$
to Eq. \eqref{gellequation}, $M(\rho):=\esssup\limits_{B_{\rho}(x_{0})}u$.
Then the following inequality holds:
\begin{equation}\label{sk 3.1}
g\left(x_0, \frac{M(\rho)+2(1+b_0)\rho}{\rho}\right)
\leqslant
\gamma\, e^{2n\lambda(\rho)}\fint\limits_{B_{5\rho/4}(x_0)}
g\left(x_0, \frac{u+2(1+b_0)\rho}{\rho}\right) dx.
\end{equation}
\end{theorem}

%Setting  ${\rm w}=\dfrac{M(\rho)+2(1+b_0)\rho}{\rho}$ and
%${\rm v}=\dfrac{m(\rho)+2(1+b_0)\rho}{\rho}$ in \eqref{gpineq}, and combining this, \eqref{sk 3.1}
%and \eqref{weakHrnckineq} (for $s=1$), we arrive at the following Harnack-type
%theorem to Eq. \eqref{gellequation}.
From Theorems \ref{thweakHarnck}, \ref{sk.th 3.1} we arrive at
\begin{theorem}[Harnack inequality]\label{sk.th 1.1}
Let all the assumptions  of Theorems \ref{thweakHarnck}, \ref{sk.th 3.1} be fulfilled.
%Fix a point $x_{0}\in \Omega$ and consider the ball $B_{8\rho}(x_{0})\in\Omega$.
%Let $u$ be a nonnegative bounded weak solution to Eq. \eqref{gellequation}, and
%let conditions {\rm (${\rm g}_{1}$)}, {\rm (${\rm g}_{2}$)} and {\rm (${\rm g}_{3}$)}
%with $K=M+2(1+b_{0})$ be fulfilled.
Then there exist positive constants $C$, $c$, $\beta$
depending only on the data such that
\begin{equation}\label{Harnckineq}
\esssup\limits_{B_{\rho}(x_{0})}u\leqslant C\Lambda(c,\beta,\rho)
\Big(\essinf\limits_{B_{\rho}(x_{0})}u+(1+b_{0})\rho\Big),
\end{equation}
where $\Lambda(c,\beta,\rho)$ was defined in Theorem~\ref{thweakHarnck}.
\end{theorem}

%The veracity of Theorem \ref{sk.th 1.1} follows from the
%weak Harnack inequality for supersolutions
%and a Moser-type sup-estimate of solutions to Eq. \eqref{gellequation}.
%These results are contained in the following two theorems.
%We apply Theorem~\ref{sk.th 1.1} to establish the point-wise continuity of solutions to
%Eq. \eqref{gellequation}.

%We note that the function $\lambda(r)=L\ln\ln\ln \dfrac{1}{r}$,
%$r\in (0,e^{-e})$, satisfies conditions \eqref{Lambd32cond}, \eqref{condonLambdabetar} if
%$0<L<1/\beta$.

%Considering specific ways of realizing the function $g(x, {\rm v})$
%by \eqref{examplfnctg12} and \eqref{eqSkr1.3},
%as corollaries of Theorems \ref{thweakHarnck}--\ref{sk.th 1.1},
%we obtain the previously known results for the $p(x)$-Laplace equation
%(see, e.g., \cite[$\lambda(r)\equiv {\rm const}$]{Alhutov97},%Theorems~1-3
%\cite[$\lambda(r)\approx L\ln\ln\ln r^{-1}$]{AlhutovKrash08, SurnPrepr2018}).
%But in the case of $(p,q)$-phase elliptic equations when the function $g(x, {\rm v})$
%is defined by \eqref{examplfnctg34}--\eqref{bxconditin} our results are new
%(cf. \cite{BarColMing, BarColMingStPt16, BarColMingCalc.Var.18, ColMing218, ColMing15}
%where the coefficients $a(x)$ and $b(x)$ are required should be H\"{o}lder continuous,
%not just continuous).

The rest of the paper contains the proofs of Theorems
\ref{thweakHarnck}, \ref{sk.th 3.1}.

\section{Proof of Theorem \ref{thweakHarnck} (the weak Harnack inequality)}

In this section we prove Theorem \ref{thweakHarnck}.
For this we need some inequalities and several lemmas.
First, we note simple analogues of Young's inequality:
\begin{equation}\label{gYoungineq1}
g(x,a)b\leqslant \varepsilon g(x,a)a+g(x,b/\varepsilon)b
\ \ \text{if} \ \varepsilon, a, b >0, \ x\in \Omega,
\end{equation}
\begin{equation}\label{gYoungineq2}
g(x,a)b\leqslant \frac{1}{\varepsilon}\,g(x,a)a
+ \varepsilon^{p-1}g(x,b)b
\ \ \text{if} \ \varepsilon\in (0,1], \ a, b >0, \ x\in \Omega.
\end{equation}
%which are valid for any $\varepsilon\in (0,1)$, $a,b>0$ and for all $x\in \Omega$.
In fact, if $b\leqslant\varepsilon a$, then $g(x,a)b\leqslant \varepsilon g(x,a)a$,
and if $b>\varepsilon a$, then since the function ${\rm v}\rightarrow g(x, {\rm v})$
is increasing we have that $g(x,a)b\leqslant g(x,b/\varepsilon)b$,
which proves inequality \eqref{gYoungineq1}.
Using assumption {\rm (${\rm g}_{2}$)} by similar arguments we arrive at 
inequality \eqref{gYoungineq2}.

Next, we set
\begin{equation}\label{defG}
\mathcal{G}(x, {\rm w}):=\int\limits_{0}^{{\rm w}}g(x,{\rm v})\,d{\rm v}
\ \ \text{for} \ \ {\rm w}>0.
\end{equation}
The following inequalities hold:
\begin{equation}\label{G>gw}
\mathcal{G}(x,{\rm w})\geqslant \gamma\,G(x,{\rm w}) \ \
\text{for all} \ \ x\in\Omega, \  {\rm w}\geqslant 2(1+b_{0}),
\end{equation}
\begin{equation}\label{gw>pG}
G(x,{\rm w})\geqslant p\,\mathcal{G}(x,{\rm w}) \ \
\text{for all} \ \ x\in \Omega, \ {\rm w}\geqslant 0.
\end{equation}
%which are consequences of the assumptions (${\rm g}_{1}$) and (${\rm g}_{2}$),
%respectively.
Indeed, if $x\in\Omega$ and ${\rm w}\geqslant 2(1+b_{0})$ then
by \eqref{gqineq}, \eqref{deffuncgw} and \eqref{defG}, we have
$$
\mathcal{G}(x,{\rm w})= \int\limits_{0}^{{\rm w}}g(x,{\rm v})\,d{\rm v}
\geqslant
\int\limits_{b_{0}}^{{\rm w}}g(x,{\rm v})\,d{\rm v}\geqslant
\frac{g(x,{\rm w})}{c_{1}{\rm w}^{\,q-1}}
\int\limits_{b_{0}}^{{\rm w}} {\rm v}^{\,q-1}d{\rm v}\geqslant
\frac{1-2^{-q}}{c_{1}q}\,G(x,{\rm w}),
$$
which implies \eqref{G>gw}.
Now, let $x\in \Omega$ and ${\rm w}\geqslant0$ be arbitrary,
then by \eqref{gpineq}, \eqref{deffuncgw} and \eqref{defG} we obtain
$$
\mathcal{G}(x,{\rm w})= \int\limits_{0}^{{\rm w}}g(x,{\rm v})\,d{\rm v}
\leqslant
\frac{g(x,{\rm w})}{{\rm w}^{p-1}}\int\limits_{0}^{{\rm w}}
{\rm v}^{p-1}\,d{\rm v}=\frac{1}{p}\,g(x,{\rm w}){\rm w}=\frac{1}{p}\,G(x,{\rm w}),
$$
which yields \eqref{gw>pG}.

The rest of the lemmas in this section are successive stages in the proof of
Theorem \ref{thweakHarnck}.
The proof follows Trudinger's strategy \cite{TrudingerArch71},
which we adapted to Eq. \eqref{gellequation}
under conditions {\rm (${\rm g}$)}, {\rm (${\rm g}_{1}$)}--{\rm (${\rm g}_{3}$)}.
\begin{lemma}\label{ShSkrlem2.1}
Let all the assumptions of Theorem \ref{thweakHarnck} be fulfilled. Then there exists positive
constant $\gamma$ depending only on the known data such that
\begin{equation}\label{ShSkr2.4}
\exp\Bigg(\fint\limits_{B_{2\rho}(x_{0})}
\ln\big(u+2(1+b_{0})\rho\big)\,dx\Bigg)
\leqslant
\Lambda(\gamma,3n,\rho) \big[m(\rho)+2(1+b_{0})\rho\big].
\end{equation}
\end{lemma}
\begin{proof}
We fix $\sigma\in(0,1)$, for any $\rho\leqslant r<r(1+\sigma)\leqslant2\rho$, we
take a function $\zeta\in C_{0}^{\infty}(B_{r(1+\sigma)}(x_{0}))$,
$0\leqslant\zeta\leqslant1$, $\zeta=1$ in $B_{r}(x_{0})$ and $|\nabla \zeta|\leqslant(\sigma r)^{-1}$.
Let
\begin{equation}\label{deffnctionw}
%w:=\bigg(\ln\frac{L}{ \overline{u} }\bigg)_{+}
%=\max\bigg\{0,\, \ln\frac{L}{\overline{u}}\bigg\},
w:=\ln\frac{\varkappa}{ \overline{u} },
\quad \overline{u}:=u+2(1+b_{0})\rho,
\end{equation}
where the constant $\varkappa$ is defined by the condition
$(w)_{x_{0},2\rho}:=\fint\limits_{B_{2\rho}(x_{0})}w\,dx=0$,
i.e.
\begin{equation}\label{overline{u}}
\varkappa:=\exp \bigg( \fint\limits_{B_{2\rho}(x_{0})}
\ln\overline{u}\,dx \bigg).
\end{equation}
We test \eqref{gelintidentity} by
%$\varphi=\dfrac{\overline{u}\,(w-k)_{+}}
%{\mathcal{G}\left(x_{0}, \overline{u}/\rho \right)}\,\zeta^{q}$,
%where $(w-k)_{+}:=\max\{0,\, w-k\}$, $k>0$
$$
\varphi=\frac{\overline{u}\,(w-k)_{+}}
{\mathcal{G}\left(x_{0}, \overline{u}/\rho \right)}\,\zeta^{q},
\quad (w-k)_{+}:=\max\{0,\, w-k\}, \quad k>0.
$$
Since we are dealing with bounded and non-negative solutions (supersolutions),
then this and all other test functions used in the paper belong to $W^{1,G}_{0}(\Omega)$.
This is a consequence of conditions {\rm (${\rm g}$)}, {\rm (${\rm g}_{1}$)}
and the result of Marcus and Mizel \cite[Theorem~2]{MarcMizel}.
So, we have
\begin{multline*}
\int\limits_{A_{k,r(1+\sigma)}}\frac{G(x,|\nabla u|)}
{\mathcal{G}\left( x_{0}, \overline{u}/\rho \right)}\,\zeta^{\,q}\,dx
+\int\limits_{A_{k,r(1+\sigma)}}\frac{G(x,|\nabla u|)}
{\mathcal{G}\left( x_{0}, \overline{u}/\rho \right)}
\left\{ \frac{G(x_{0}, \overline{u}/\rho)}
{\mathcal{G}(x_{0}, \overline{u}/\rho)}-1 \right\}(w-k)_{+}\,\zeta^{q}\,dx
\\
\leqslant \frac{\gamma}{\sigma}\int\limits_{A_{k,r(1+\sigma)}}
\frac{g(x,|\nabla u|)}
{\mathcal{G}\left( x_{0}, \overline{u}/\rho \right)}\,
\frac{\overline{u}}{\rho}\,(w-k)_{+}\,\zeta^{q-1}\,dx,
\end{multline*}
here $A_{k,r}:=B_{r}(x_{0})\cap \{w>k\}$. %and
%$A_{k,r, r(1+\sigma)}:=\big(B_{r(1+\sigma)}(x_{0})\setminus B_{r}(x_{0})\big)\cap \{w>k\}$.
By \eqref{gw>pG}, the value in curly brackets is estimated from below as follows:
\begin{equation}\label{GmthclGp-1}
\frac{G(x_{0}, \overline{u}/\rho)}
{\mathcal{G}(x_{0}, \overline{u}/\rho)}-1\geqslant p-1,
\end{equation}
and therefore
\begin{multline}\label{Voit1}
\int\limits_{A_{k,r(1+\sigma)}}\frac{G(x,|\nabla u|)}
{\mathcal{G}\left( x_{0}, \overline{u}/\rho \right)}\,\zeta^{q}\,dx+
(p-1) \int\limits_{A_{k,r(1+\sigma)}}\frac{G(x,|\nabla u|)}
{\mathcal{G}\left( x_{0}, \overline{u}/\rho \right)}\,(w-k)_{+}\,\zeta^{q}\,dx
\\
\leqslant \gamma\int\limits_{A_{k,r(1+\sigma)}}
\frac{g(x,|\nabla u|)}
{\mathcal{G}\left( x_{0}, \overline{u}/\rho \right)}\,
\frac{\overline{u}}{\sigma\rho\,\zeta}\,(w-k)_{+}\,\zeta^{q}\,dx.
\end{multline}
We use inequality \eqref{gYoungineq1} with $a=|\nabla u|$,
$b=\dfrac{\overline{u}}{\sigma\rho\,\zeta}$ and
sufficiently small $\varepsilon>0$, and then \eqref{G>gw} with
${\rm w}=\overline{u}/\rho$,
to estimate from above the right-hand side of \eqref{Voit1}:
\begin{multline*}\label{Voit1}
\gamma\int\limits_{A_{k,r(1+\sigma)}}
\frac{g(x,|\nabla u|)}
{\mathcal{G}\left( x_{0}, \overline{u}/\rho \right)}\,
\frac{\overline{u}}{\sigma\rho\,\zeta}\,(w-k)_{+}\,\zeta^{q}\,dx
\\
\leqslant
\frac{p-1}{2\gamma}\int\limits_{A_{k,r(1+\sigma)}}
\frac{G(x,|\nabla u|)}
{\mathcal{G}\left( x_{0}, \overline{u}/\rho \right)}\,
(w-k)_{+}\,\zeta^{q}\,dx
+\frac{\gamma}{\sigma}\int\limits_{A_{k,r(1+\sigma)}}
\frac{g\big( x,  \frac{\gamma\,\overline{u}}{\sigma\rho\,\zeta} \big)}
{g\left( x_{0}, \overline{u}/\rho \right)}(w-k)_{+}\,\zeta^{q-1}\,dx.
\end{multline*}
Combining this inequality and \eqref{Voit1}, %and using \eqref{gw>pG} with ${\rm w}=|\nabla u|$,
we obtain that
\begin{equation}\label{ineqGGgg}
\int\limits_{A_{k,r(1+\sigma)}}
\frac{G(x,|\nabla u|)}{\mathcal{G}\left( x_{0}, \overline{u}/\rho \right)}\,\zeta^{q}\,dx
\leqslant
\frac{\gamma}{\sigma}\int\limits_{A_{k,r(1+\sigma)}}
\frac{g\big( x,  \frac{\gamma\,\overline{u}}{\sigma\rho\,\zeta} \big)}
{g\left( x_{0}, \overline{u}/\rho \right)}\,(w-k)_{+}\,\zeta^{q-1}\,dx.
\end{equation}
Since
$
 \dfrac{\gamma\,\overline{u}}{\sigma\rho\,\zeta}
\geqslant
\dfrac{\overline{u}}{\rho}
\geqslant b_{0}
$ and $|x-x_{0}|<r(1+\sigma)\leqslant 2\rho$ for $x\in A_{k,r(1+\sigma)}$,
then using conditions (${\rm g}_{1}$) and (${\rm g}_{3}$), we get
$$
g\left(x,  \frac{\gamma\,\overline{u}}{\sigma\rho\,\zeta}\right)
\leqslant\gamma\,(\sigma\zeta)^{1-q}\,
g\left(x,\overline{u}/\rho\right)
\leqslant
\gamma\,(\sigma\zeta)^{1-q}\,e^{\lambda(\rho)}
g\left(x_{0},\overline{u}/\rho \right)
\ \ \text{for all} \  x\in  A_{k,r(1+\sigma)}.
$$
So, from \eqref{ineqGGgg} we obtain
\begin{equation}\label{ShSkr2.5}
\int\limits_{A_{k,r(1+\sigma)}}
\frac{G(x,|\nabla u|)}{\mathcal{G}\left( x_{0}, \overline{u}/\rho \right)}\,\zeta^{q}\,dx
\leqslant
\gamma\,\sigma^{-q}\,e^{\lambda(\rho)}
\int\limits_{A_{k,r(1+\sigma)}} (w-k)_{+}\,dx.
\end{equation}

To estimate the term on the left-hand side of \eqref{ShSkr2.5}, we use \eqref{gYoungineq1} with
$\varepsilon\!=\!1$, $a=\overline{u}/\rho$, $b=|\nabla u|$, assumption
(${\rm g}_{3}$), the definitions of the functions $G$, $\mathcal{G}$, $w$
(see equalities \eqref{deffuncgw}, \eqref{defG} and \eqref{deffnctionw}, respectively)
and \eqref{gw>pG}:
\begin{equation}\label{eqShSkr2.6}
\begin{aligned}
\int\limits_{A_{k,r(1+\sigma)}}|\nabla w|\,\zeta^{q}\,dx
&=
\int\limits_{A_{k,r(1+\sigma)}} \frac{|\nabla u|}{\overline{u}}\,
\frac{g\left(x,\overline{u}/\rho\right)}{g\left(x,\overline{u}/\rho\right)}\,\zeta^{q}\,dx
\\
&\leqslant \frac{1}{\rho}\,|A_{k,r(1+\sigma)}|+
\frac{1}{\rho}\int\limits_{A_{k,r(1+\sigma)}}
\frac{G(x,|\nabla u|)}{G\left( x, \overline{u}/\rho \right)}\,\zeta^{q}\,dx
\\
&\leqslant \frac{1}{\rho}\,|A_{k,r(1+\sigma)}|+
\gamma \frac{e^{\lambda(\rho)}}{\rho}
\int\limits_{A_{k,r(1+\sigma)}}
\frac{G(x,|\nabla u|)}{\mathcal{G}\left( x_{0}, \overline{u}/\rho \right)}\,\zeta^{q}\,dx.
\end{aligned}
\end{equation}
Collecting \eqref{ShSkr2.5} and \eqref{eqShSkr2.6}, we obtain
$$
\int\limits_{A_{k,r(1+\sigma)}}|\nabla w|\,\zeta^{q}\,dx\leqslant
\frac{\gamma\,e^{2\lambda(\rho)}}{\sigma^{q}\rho}
\bigg(|A_{k,r(1+\sigma)}|+ \int\limits_{A_{k,r(1+\sigma)}} (w-k)_{+}\,dx\bigg).
$$
From this, using Sobolev embedding theorem and standard iteration arguments
(see, for instance \cite[Section~2, Theorem~5.3]{LadUr}),
and choosing $k$ from the condition
$$
k=\gamma\,e^{2n\lambda(\rho)}
\Bigg(\fint\limits_{B_{2\rho}(x_{0})}
|w|^{\,\frac{n}{n-1}}\,dx\Bigg)^{\frac{n-1}{n}}+1,
$$
we obtain that
\begin{equation}\label{ShSkr2.7}
\esssup\limits_{B_{\rho}(x_{0})}w\leqslant
\gamma\,e^{2n\lambda(\rho)}
\Bigg(\fint\limits_{B_{2\rho}(x_{0})}
|w|^{\,\frac{n}{n-1}}\,dx\Bigg)^{\frac{n-1}{n}}+1.
\end{equation}

To estimate the right-hand side of \eqref{ShSkr2.7} we use the Poincar\'{e} inequality,
by our choice of $\varkappa$ (see \eqref{overline{u}}) we have
\begin{equation}\label{Poincarewn/n-1}
\Bigg(\fint\limits_{B_{2\rho}(x_{0})}
|w|^{\,\frac{n}{n-1}}\,dx\Bigg)^{\frac{n-1}{n}}
=\Bigg(\fint\limits_{B_{2\rho}(x_{0})}
|w-(w)_{x_{0},2\rho}|^{\,\frac{n}{n-1}}\,dx\Bigg)^{\frac{n-1}{n}}
\leqslant \gamma\,\rho^{1-n}
\int\limits_{B_{2\rho}(x_{0})}|\nabla w|\,dx.
\end{equation}
Next, similarly to \eqref{eqShSkr2.6}, we have
\begin{equation}\label{intDwbyintGG}
\int\limits_{B_{2\rho}(x_{0})}|\nabla w|\,dx\leqslant
\int\limits_{B_{4\rho}(x_{0})}|\nabla w|\,\zeta^{q}\,dx
\leqslant \gamma\rho^{n-1}+\gamma\,\frac{e^{\lambda(\rho)}}{\rho}
\int\limits_{B_{4\rho}(x_{0})}
\frac{G(x,|\nabla u|)}
{\mathcal{G}\left( x_{0}, \overline{u}/\rho \right)}\,\zeta^{\,q}\,dx,
\end{equation}
here $\zeta\in C_{0}^{\infty}(B_{4\rho}(x_{0}))$,
$0\leqslant\zeta\leqslant1$, $\zeta=1$ in $B_{2\rho}(x_{0})$, and
$|\nabla\zeta|\leqslant 2/\rho$.
In addition, testing \eqref{gelintidentity} by
$
\varphi=\dfrac{\overline{u}\,\zeta^{q}}{\mathcal{G}\left( x_{0}, \overline{u}/\rho\right)}
$,
similarly to \eqref{ShSkr2.5}, we obtain
\begin{equation}\label{ShSkrVoit2.8}
\int\limits_{B_{4\rho}(x_{0})}
\frac{G(x,|\nabla u|)}{\mathcal{G}\left( x_{0}, \overline{u}/\rho \right)}\,\zeta^{q}\,dx
\leqslant \gamma \rho^{n}e^{\lambda(\rho)}.
\end{equation}
Now, collecting \eqref{ShSkr2.7}--\eqref{ShSkrVoit2.8} and taking into account
\eqref{deffnctionw} and \eqref{overline{u}},
%\eqref{ShSkr2.7}, \eqref{Poincarewn/n-1}, \eqref{intDwbyintGG},
%\eqref{ShSkrVoit2.8}
we arrive at the required \eqref{ShSkr2.4}.
The proof of the lemma is complete.
\end{proof}
%%%%%%%%%%%%%%%%%%%%%%%%%%%%%%%%%%%%%%%%%%%%%%%%%%%%%%%%%%%%%%%%%%%%%%%%%%%%%%%%%%%%%%%%%%%%%%%%%%%%%%%%%%%%%%%%%%%
\begin{lemma}\label{ShSkrlem2.2}
Under the assumptions of Theorem \ref{thweakHarnck} there exists
$\delta_{0}=\delta_{0}(\rho)>0$ depending only on the data and $\rho$ such that
\begin{multline}\label{eqShSkr2.9}
\Bigg(\fint\limits_{B_{3\rho/2}(x_{0})}
\big(u+2(1+b_{0})\rho\big)^{\delta_{0}}dx\Bigg)^{1/\delta_{0}}
\leqslant \Lambda(\gamma,2n-1,\rho)\exp
\Bigg( \fint\limits_{B_{2\rho}(x_{0})}
\ln\big(u+2(1+b_{0})\rho\big)\,dx\Bigg).
\end{multline}
\end{lemma}
\begin{proof}
Let's fix $\sigma\in(0,1)$ and for any $3\rho/2\leqslant r<r(1+\sigma)\leqslant2\rho$
consider the function
$$
\zeta\in C_{0}^{\infty}\big(B_{r(1+\sigma)}(x_{0})\big), \quad
0\leqslant\zeta\leqslant1, \quad \zeta=1 \ \text{in} \ B_{r}(x_{0}),
\quad |\nabla \zeta|\leqslant(\sigma r)^{-1}.
$$
We define the functions
$$
v:=\ln\frac{u+2(1+b_{0})\rho}{\varkappa}=\ln\frac{\overline{u}}{\varkappa}, \quad
v_{\mu}:=\max\{v,\,\mu\}, \quad \mu>0.
$$
Testing \eqref{gelintidentity} by
$$
\varphi=\frac{v_{\mu}^{s-1}\,\overline{u}\,\zeta^{\,l}}
{\mathcal{G}\left( x_{0}, \overline{u}/\rho \right)},
\ \ s\geqslant1, \ \ l\geqslant q,
$$
and using \eqref{GmthclGp-1}, we have
\begin{multline*}
(p-1)\int\limits_{B_{r(1+\sigma)}(x_{0})}
\frac{G(x,|\nabla u|)}{\mathcal{G}\left( x_{0}, \overline{u}/\rho \right)}\,
v_{\mu}^{s-1}\,\zeta^{\,l}\,dx
\\
\leqslant
(s-1)\int\limits_{B_{r(1+\sigma)}(x_{0})\cap\{v>\mu\}}
\frac{G(x,|\nabla u|)}{\mathcal{G}\left( x_{0}, \overline{u}/\rho \right)}\,
v_{\mu}^{s-2}\,\zeta^{\,l}\,dx
+\gamma\, l\int\limits_{B_{r(1+\sigma)}(x_{0})}
\frac{g(x,|\nabla u|)}
{\mathcal{G}\left( x_{0}, \overline{u}/\rho \right)}\,
\frac{\overline{u}}{\sigma\rho\,\zeta}\,
 v_{\mu}^{s-1}\,\zeta^{\,l}\,dx.
\end{multline*}
Choosing $\mu$ from the condition $\dfrac{s}{\mu}=\dfrac{p-1}{2}$
and using inequalities \eqref{gYoungineq1}, \eqref{G>gw} and conditions (${\rm g}_{1}$)
and (${\rm g}_{3}$) similarly to the derivation of \eqref{ShSkr2.5},
from the previous we obtain
\begin{equation}\label{ShSkr2.10}
\int\limits_{B_{r(1+\sigma)}(x_{0})}
\frac{G(x,|\nabla u|)}{\mathcal{G}\left( x_{0}, \overline{u}/\rho \right)}\,
v_{\mu}^{s-1}\,\zeta^{\,l}\,dx
\leqslant \frac{\gamma\,l^{\gamma}e^{\lambda(\rho)}}{\sigma^{q}}
\int\limits_{B_{r(1+\sigma)}(x_{0})} v_{\mu}^{s-1}\,\zeta^{\,l-q}\,dx.
\end{equation}
Estimating the term on the left-hand side of \eqref{ShSkr2.10}, similarly to \eqref{eqShSkr2.6},
we obtain
\begin{multline*}
\int\limits_{B_{r(1+\sigma)}(x_{0})}
|\nabla v_{\mu}|\,v_{\mu}^{s-1}\,\zeta^{\,l}\,dx
\leqslant
\int\limits_{B_{r(1+\sigma)}(x_{0})}
\frac{|\nabla u|}{\overline{u}}\,v_{\mu}^{s-1}\,\zeta^{\,l}\,dx
\\
\leqslant
\frac{\gamma\,l^{\gamma}}{\sigma^{q}}\,\frac{e^{2\lambda(\rho)}}{\rho}
\int\limits_{B_{r(1+\sigma)}(x_{0})} v_{\mu}^{s-1}\,\zeta^{\,l-q}\,dx
\leqslant
\frac{\gamma\,l^{\gamma}}{\sigma^{q}}\,\frac{e^{2\lambda(\rho)}}{\rho}
\int\limits_{B_{r(1+\sigma)}(x_{0})} v_{\mu}^{s}\,\zeta^{\,l-q}\,dx.
\end{multline*}
Using Sobolev's embedding theorem from this we have
\begin{equation}\label{ShSkr2.11}
\fint\limits_{B_{r}(x_{0})}
v_{\mu}^{\frac{sn}{n-1}}dx\leqslant
\Bigg( \frac{\gamma s\,e^{2\lambda(\rho)}}{\sigma^{q}} %\gamma s\,\sigma^{-q}\,e^{2\lambda(\rho)}
\fint\limits_{B_{r(1+\sigma)}(x_{0})}v_{\mu}^{s}\, dx\Bigg)^{\frac{n}{n-1}}.
\end{equation}

For $j=0,1,2,\ldots$, we define the following sequences:
$r_{j}:=\dfrac{\rho}{2}(3+2^{-j})$, $B_{j}:=B_{r_{j}}(x_{0})$,
$$
s_{j}:=\bigg(\frac{n}{n-1}\bigg)^{j+1},\quad
\mu_{j}:=\frac{2s_{j}}{p-1}, \quad
y_{j}:=\fint\limits_{B_{j}}
v_{\mu_{j}}^{s_{j}}\,dx.
$$
Then inequality \eqref{ShSkr2.11} can be rewritten in the form
\begin{equation}\label{ShSkr2.12}
y_{j+1}\leqslant\gamma^{\frac{n}{n-1}}\,2^{\frac{j\gamma n}{n-1}}\,
s_{j}^{\frac{n}{n-1}}\,e^{\frac{2n}{n-1}\,\lambda(\rho)}\,y_{j}^{\frac{n}{n-1}},
\quad j=0,1,2,\ldots\,,
\end{equation}
and by \eqref{ShSkrVoit2.8}
\begin{equation}\label{ShSkr2.13}
y_{0}\leqslant\gamma\,e^{\lambda(\rho)}.
\end{equation}
From this by iteration, we have for $j=0,1,2,\ldots$
\begin{equation}\label{ShSkr2.14}
y_{j+1}^{1/s_{j+1}}
\leqslant\gamma^{\,\sum\limits_{i=0}^{j} \frac{1}{s_{i}}}\,
2^{\gamma\sum\limits_{i=1}^{j}\frac{i}{s_{i}}}
\Big(\frac{n}{n-1}\Big)^{\sum\limits_{i=0}^{j}\frac{i+1}{s_{i}}}
\exp\bigg\{2\lambda(\rho) \sum\limits_{i=0}^{j}\frac{1}{s_{i}} \bigg\}\,
y_{0}^{\frac{n-1}{n}}
\leqslant \gamma\, e^{(2n-1)\lambda(\rho)}.
\end{equation}
%\begin{equation}
%y_{j}^{1/s_{j}}\leqslant\gamma^{1+\sum\limits_{i=0}^{j-1}1/s_{i}}\,2^{\gamma\sum\limits_{i=0}^{j-1}i/s_{i}}
%s_{j}\exp\bigg\{\bigg(1+2\sum\limits_{i=0}^{j-1}1/s_{i}\bigg)\lambda(\rho)\bigg\}
%\leqslant \gamma e^{(2n+1)\lambda(\rho)}s_{j}\,.
%\end{equation}
Let $m\in \mathbb{N}$ be arbitrary, then there exists $j\geqslant 1$ such that
$s_{j-1}<m\leqslant s_{j}$.
Using H\"{o}lder's inequality, %and Stirling's ($m!>\sqrt{2\pi m}\,(m/e)^{m}$)
from \eqref{ShSkr2.14} we obtain
$$
\fint\limits_{B_{3\rho/2}(x_{0})}
\frac{v_{+}^{\,m}}{m!}\,dx
\leqslant
\fint\limits_{B_{3\rho/2}(x_{0})}
\frac{v_{\mu_{j}}^{\,m}}{m!}\,dx
\leqslant
\frac{\gamma\,y_{j}^{m/s_{j}}}{m!}
\leqslant \frac{\gamma^{m+1}}{m!}\,e^{(2n-1)m\lambda(\rho)}
\leqslant \gamma^{m+1}e^{(2n-1)m\lambda(\rho)}.
$$
Choosing $\delta_{0}=\delta_{0}(\rho)$ from the condition
\begin{equation}\label{ShSkr2.15}
\delta_{0}:=\frac{1}{2\gamma}\,e^{-(2n-1)\lambda(\rho)},
\end{equation}
from the previous we have
$$
\fint\limits_{B_{3\rho/2}(x_{0})}
\frac{(\delta_{0}v_{+})^{m}}{m\,!}\,dx\leqslant \gamma\,2^{-m},
$$
which implies that
$$
\fint\limits_{B_{3\rho/2}(x_{0})}
e^{\delta_{0}v}\,dx
\leqslant
\fint\limits_{B_{3\rho/2}(x_{0})}
e^{\delta_{0}v_{+}}\,dx
\leqslant
\sum\limits_{m=0}^{\infty}\,
\fint\limits_{B_{3\rho/2}(x_{0})}
\frac{(\delta_{0}v_{+})^{m}}{m!}\,dx\leqslant 2\gamma.
$$
From this, since
$e^{\delta_{0}v}= (\overline{u}/\varkappa)^{\delta_{0}}$
we have
$$
\Bigg(\fint\limits_{B_{3\rho/2}(x_{0})}
\overline{u}^{\,\delta_{0}}\,dx\Bigg)^{1/\delta_{0}}
\leqslant (2\gamma)^{1/\delta_{0}} \varkappa
\leqslant \Lambda(\gamma, 2n-1, \rho)\,\varkappa,
$$
that together with \eqref{overline{u}} yields the desired inequality \eqref{eqShSkr2.9}.
This completes the proof of the lemma.
\end{proof}

The next lemma is a simple consequence of Lemmas \ref{ShSkrlem2.1}
and \ref{ShSkrlem2.2}.

\begin{lemma}\label{lemShSkrVoit2.3}
Let all the assumptions of Lemma \ref{ShSkrlem2.2} be fulfilled,
and set
\begin{equation}\label{defdelta1}
\delta_{1}:=\delta_{0}/(q-1),
\end{equation}
where $\delta_{0}$ is defined by \eqref{ShSkr2.15}.
Then the following inequality holds:
\begin{multline}\label{ShSkr2.16}
\Bigg(\fint\limits_{B_{3\rho/2}(x_{0})}
g^{\delta_{1}}\left(x_{0}, \frac{u+2(1+b_{0})\rho}{\rho}\right)\,dx\Bigg)
^{1/\delta_{1}}
\leqslant \Lambda(\gamma, 3n, \rho)\,
g\left(x_{0}, \frac{m(\rho)+2(1+b_{0})\rho}{\rho}\right).
\end{multline}
\end{lemma}

\begin{proof}
By condition (${\rm g}_{1}$) we have
\begin{equation*}
\fint\limits_{B_{3\rho/2}(x_{0})}
\frac{g^{\delta_{1}}\left(x_{0}, \frac{u+2(1+b_{0})\rho}{\rho}\right)}
{g^{\delta_{1}}\left(x_{0}, \frac{m(\rho)+2(1+b_{0})\rho}{\rho}\right)}\,dx
\leqslant 1+ c_{1}^{\delta_{1}}
\fint\limits_{B_{3\rho/2}(x_{0})\cap\{u>m(\rho)\}}
\left(\frac{u+2(1+b_{0})\rho}{m(\rho)+2(1+b_{0})\rho}\right)
^{\delta_{0}}\,dx.
\end{equation*}
By Lemmas \ref{ShSkrlem2.1} and \ref{ShSkrlem2.2} the second term on the right-hand side of
this inequality is estimated from above as follows:
$$
\fint\limits_{B_{3\rho/2}(x_{0})}
\left(\frac{u+2(1+b_{0})\rho}{m(\rho)+2(1+b_{0})\rho}\right)
^{\delta_{0}}\,dx\leqslant \Lambda(\gamma,3n,\rho),
$$
which proves the lemma.
\end{proof}

To complete the proof of Theorem \ref{thweakHarnck} we need the following lemma.
\begin{lemma}[inverse H\"{o}lder inequality]\label{sk.lem 2.4}
Let the assumptions  of Theorem \ref{thweakHarnck} be fulfilled,
then for all $\delta_1\leqslant s<n/(n-1)$ the following inequality holds:
\begin{multline}\label{sk 2.17}
\Bigg(\fint\limits_{B_{5\rho/4}(x_0)} g^s\left(x_0, \frac{u+2(1+b_0)\rho}{\rho}\right) dx
\Bigg)^{1/s}\\
\leqslant
\Lambda(\gamma, 2n, \rho)
%\bigg(1-\frac{n-1}{n}s\bigg)
%^{-\gamma e^{2n\lambda(\rho)}}
\Bigg(\fint\limits_{B_{3\rho/2}(x_0)} g^{\delta_1}
\left(x_0, \frac{u+2(1+b_0)\rho}{\rho}\right) dx \Bigg)^{1/\delta_1}.
\end{multline}
\end{lemma}
\begin{proof}
%To prove (\ref{sk 2.17})
We set
$$
\psi(x, {\rm w}):= \frac{\mathcal{G}(x, {\rm w})}{{\rm w}}  \ \
\text{ for } \ x\in\Omega, \  {\rm w}>0,
$$
and note that by \eqref{G>gw} and  \eqref{gw>pG}, we have
\begin{equation}\label{g<gammapsi}
g(x,{\rm w})\leqslant \gamma\,\psi(x,{\rm w}) \ \
\text{for all} \ \ x\in\Omega, \  {\rm w}\geqslant 2(1+b_{0}),
\end{equation}
\begin{equation}\label{psi<g}
\psi(x, {\rm w})\leqslant p^{-1} g(x, {\rm w}) \ \
\text{for all} \ \ x\in \Omega, \ {\rm w}\geqslant 0,
\end{equation}
which gives
\begin{equation}\label{psi'<psi/w}
\psi'_{{\rm w}} (x, {\rm w})\leqslant \gamma\,
\frac{\psi(x, {\rm w})}{{\rm w}} \ \
\text{for all} \ \ x\in\Omega, \  {\rm w}\geqslant 2(1+b_{0}),
\end{equation}
\begin{equation}\label{psi'>psi/w}
\psi'_{{\rm w}} (x, {\rm w})
=\frac{g(x, {\rm w})-\psi(x, {\rm w})}{{\rm w}}
\geqslant (p-1)\, \frac{\psi(x, {\rm w})}{{\rm w}}
\ \
\text{for all} \ \ x\in \Omega, \ {\rm w}> 0.
\end{equation}

We need a Cacciopoli-type inequality for negative powers of
$\psi\left(x_0, \overline{u}/\rho\right)$.
To establish it, we fix $\sigma\in (0, 1)$ and $r>0$ such that
$5\rho/4 \leqslant r< r(1+\sigma)\leqslant 3\rho/2$, and
take a function
$$
\zeta\in C_0^{\infty}\left(B_{r(1+\sigma)}(x_0)\right), \quad
0\leqslant \zeta \leqslant 1, \quad
\zeta=1 \text{ in } B_{r}(x_0), \quad
|\nabla \zeta|\leqslant (\sigma r)^{-1}.
$$
Testing \eqref{gelintidentity} by
$$
\varphi:= \psi^{-\tau}\left(x_0, \overline{u}/\rho\right)\zeta^{\,\theta},
\ \ 0<\tau<1, \ \ \theta\geqslant q,
$$
and using \eqref{psi'>psi/w}, we obtain
\begin{equation*}
\begin{aligned}
(p-1)\,\tau \int\limits_{B_{r(1+\sigma)}(x_0)}
&\psi^{-\tau}\left(x_0, \overline{u}/\rho\right)
\frac{G(x, |\nabla u|)}{\overline{u}}\,\zeta^{\,\theta}\, dx
\\
\leqslant\frac{\gamma\, \theta }{\sigma\rho}\int\limits_{B_{r(1+\sigma)}(x_0)}
&\psi^{-\tau} \left(x_0, \overline{u}/\rho\right) g(x, |\nabla u|)\,\zeta^{\,\theta-1} dx,
\end{aligned}
\end{equation*}
which implies by \eqref{gYoungineq1}, (${\rm g}_{1}$), (${\rm g}_{3}$) and \eqref{g<gammapsi} that
\begin{multline}\label{sk 2.18}
\int\limits_{B_{r(1+\sigma)}(x_0)}
\psi^{-\tau}\left(x_0,\overline{u}/\rho\right)
\frac{G(x,|\nabla u|)}{\overline{u}}\,\zeta^{\,\theta}\,dx
 \\
 \leqslant
\frac{\gamma\, \theta^{\,q}}{(\sigma \tau)^{q}}\,
\frac{e^{\lambda(\rho)}}{\rho}
\int\limits_{B_{r(1+\sigma)}(x_0)} \psi^{1-\tau}
\left(x_0, \overline{u}/\rho\right) \,\zeta^{\,\theta-q}\, dx.
\end{multline}

Based on the inequality \eqref{sk 2.18}, we organize Moser-type iterations
for the function $\psi\left(x_0, \overline{u}/\rho\right)$.
To do this, we fix $0<t<n/(n-1)$ and $l\geqslant nq/(n-1)$,
then by the Sobolev inequality and by \eqref{psi'<psi/w} and \eqref{psi<g}, we obtain
\begin{equation}\label{sk 2.19}
\begin{aligned}
\Bigg(\int\limits_{B_{r(1+\sigma)}(x_0)}
\psi^{\,t}\left(x_0, \overline{u}/\rho\right)\,\zeta^{\,l}\, dx\Bigg)^{\frac{n-1}{n}}
\leqslant \gamma  \int\limits_{B_{r(1+\sigma)}(x_0)}
\left| \nabla \Big[\psi^{\,\frac{t(n-1)}{n}}\left(x_0, \overline{u}/\rho\right)
\zeta^{\,\frac{l(n-1)}{n}}\Big] \right|\,dx
 \\
 \leqslant\gamma t\int\limits_{B_{r(1+\sigma)}(x_0)}
 \psi^{\,\frac{t(n-1)}{n}-1}\left(x_0, \overline{u}/\rho\right)
 \frac{g\left(x_0, \overline{u}/\rho\right)}
 {\overline{u}}\,|\nabla u|\,\zeta^{\,\frac{l(n-1)}{n}}\, dx
 \\
+\frac{\gamma\, l}{\sigma\rho}\int\limits_{B_{r(1+\sigma)}(x_0)}
\psi^{\,\frac{t(n-1)}{n}}
\left(x_0, \overline{u}/\rho\right)\zeta^{\,\frac{l(n-1)}{n}-1}\, dx.
\end{aligned}
\end{equation}
Using \eqref{gYoungineq1}, (${\rm g}_{3}$), \eqref{g<gammapsi} and (\ref{sk 2.18})
with $\tau=1-t(n-1)/n$ and $\theta=l(n-1)/n$,
we estimate the first term on the right-hand side of (\ref{sk 2.19}) as follows:
\begin{equation}\label{sk 2.20}
\begin{aligned}
&\int\limits_{B_{r(1+\sigma)}(x_0)} \psi^{\,\frac{t(n-1)}{n}-1}
\left(x_0, \overline{u}/\rho\right)
\frac{g\left(x_0, \overline{u}/\rho\right)}{\overline{u}}\,
|\nabla u|\,\zeta^{\,\frac{l(n-1)}{n}} dx
\\
&\leqslant\gamma e^{\lambda(\rho)}\int\limits_{B_{r(1+\sigma)}(x_0)}
\psi^{\,\frac{t(n-1)}{n}-1}\left(x_0, \overline{u}/\rho\right)
\frac{g\left(x,\overline{u}/\rho\right)}{\overline{u}}\,
|\nabla u|\,\zeta^{\,\frac{l(n-1)}{n}} dx
\\
&\leqslant\gamma e^{\lambda(\rho)}\int\limits_{B_{r(1+\sigma)}(x_0)}
\psi^{\,\frac{t(n-1)}{n}-1}\left(x_0, \overline{u}/\rho\right)
\frac{G\left(x, |\nabla u|\right)}{\overline{u}}\,\zeta^{\,\frac{l(n-1)}{n}} dx
\\
&+\gamma\,\frac{e^{\lambda(\rho)}}{\rho} \int\limits_{B_{r(1+\sigma)}(x_0)}
\psi^{\,\frac{t(n-1)}{n}-1}\left(x_0, \overline{u}/\rho\right)
g\left(x, \overline{u}/\rho\right)\,\zeta^{\,\frac{l(n-1)}{n}} dx
\\
&\leqslant\frac{\gamma\,l^{q}}{\sigma^{q}}
\left( 1-\frac{t(n-1)}{n} \right)^{-q}
\frac{e^{2\lambda(\rho)}}{\rho}\int\limits_{B_{r(1+\sigma)}(x_0)} \psi^{\,\frac{t(n-1)}{n}}
\left(x_0, \overline{u}/\rho\right) \,\zeta^{\,\frac{l(n-1)}{n}-q}\, dx.
\end{aligned}
\end{equation}
Combining \eqref{sk 2.19}, \eqref{sk 2.20},  we arrive at
\begin{equation}\label{sk 2.21}
\begin{aligned}
\Bigg(\fint\limits_{B_r}
&\psi^{\,t}\left(x_0, \overline{u}/\rho\right) dx\Bigg)^{\frac{n-1}{n}}
\leqslant\frac{\gamma\,l^{q}}{\sigma^{q}}
\left( 1-\frac{t(n-1)}{n} \right)^{-q}
e^{2\lambda(\rho)}
\\
&\times\fint\limits_{B_{r(1+\sigma)}(x_0)}
\psi^{\,\frac{t(n-1)}{n}} \left(x_0, \overline{u}/\rho\right) dx,
\quad
0<t<\frac{n}{n-1}, \quad l\geqslant \frac{nq}{n-1}.
\end{aligned}
\end{equation}
Now, let $\delta_{1}\leqslant s<n/(n-1)$,
and let $j$ be a non-negative integer number such that
\begin{equation}\label{tj+1<d<tj}
s\left(\frac{n-1}{n}\right)^{j+1} \leqslant\delta_1\leqslant s\left(\frac{n-1}{n}\right)^{j}.
\end{equation}
Setting in \eqref{sk 2.21} $l:=nq$, $r=r_{i}:= \dfrac{\rho}{4}(6-2^{-i})$,
$r(1+\sigma)=r_{i+1}$,
$B_{i}:= B_{r_{i}}(x_0)$ and $t=t_{i}:= s\left(\frac{n-1}{n}\right)^{i}$ for $i=0, 1,\ldots,j+1$,
we have
\begin{equation*}
\Bigg(\fint\limits_{B_{i}}
\psi^{\,t_{i}}\left(x_0, \overline{u}/\rho\right) dx\Bigg)^{\frac{1}{t_{i}}}
\leqslant
\left[\gamma\, 2^{iq}
\left( 1-\frac{n-1}{n}s \right)^{-q}
e^{2\lambda(\rho)} \right]^{\frac{1}{t_{i+1}}}
\Bigg(\fint\limits_{B_{i+1}}
\psi^{\,t_{i+1}}\left(x_0, \overline{u}/\rho\right) dx\Bigg)^{\frac{1}{t_{i+1}}}.
\end{equation*}
Iterating this relation and using H\"{o}lder's inequality, we obtain
\begin{equation*}
\begin{aligned}
\Bigg(&\fint\limits_{B_{5\rho/4}(x_0)} \psi^{\,s}
\left(x_0, \overline{u}/\rho\right) dx\Bigg)^{\frac{1}{s}}
= \Bigg(  \fint\limits_{B_{0}}
\psi^{\,t_{0}}\left(x_0, \overline{u}/\rho\right)dx \Bigg)^{\frac{1}{t_{0}}}
\\
&\leqslant\prod\limits_{i=0}^j\left[ \gamma\, 2^{i\gamma} e^{2\lambda(\rho)}
\bigg(1-\frac{n-1}{n}s \bigg)^{-q}
\right]^{\frac{1}{t_{i+1}}}\Bigg(\fint\limits_{B_{j+1}} \psi^{\,t_{j+1}}
\left(x_0, \overline{u}/\rho\right) dx\Bigg)^{\frac{1}{t_{j+1}}}
\\
&\leqslant
2^{\gamma \sum\limits_{i=0}^{j}i/t_{i+1}}
\left[ \gamma\, e^{2\lambda(\rho)}
\bigg(1-\frac{n-1}{n}s \bigg)^{-q}
\right]^{\sum\limits_{i=0}^{j}1/t_{i+1}}
\Bigg(\gamma \fint\limits_{B_{3\rho/2}(x_0)} \psi^{\,\delta_1}
\left(x_0, \overline{u}/\rho\right) dx\Bigg)^{\frac{1}{\delta_1}},
\end{aligned}
\end{equation*}
and by \eqref{tj+1<d<tj}, \eqref{defdelta1} and \eqref{ShSkr2.15}
$$
\sum\limits_{i=0}^{j}\frac{1}{t_{i+1}}
\leqslant
\frac{1}{\delta_{1}}\, \frac{n}{n-1}
\sum\limits_{i=0}^{\infty} \left(\frac{n-1}{n}\right)^{i}
=\frac{n^{2}}{\delta_{1}(n-1)},
$$
$$
\sum\limits_{i=0}^{j}\frac{i}{t_{i+1}}\leqslant
j \sum\limits_{i=0}^{j}\frac{1}{t_{i+1}}\leqslant
\frac{\gamma (\lambda(\rho)+1)}{\delta_{1}}.
$$
From this, recalling the definition of $\delta_1$
(see again \eqref{defdelta1} and \eqref{ShSkr2.15}),
we arrive at the required \eqref{sk 2.17}.
This completes the proof of the lemma.
\end{proof}

Combining Lemmas \ref{lemShSkrVoit2.3} and \ref{sk.lem 2.4}, we obtain that
\begin{equation*}
\Bigg(\fint\limits_{B_{5\rho/4}(x_0)}
g^{\,s}\left(x_0, \frac{u+2(1+b_0)\rho}{\rho}\right) dx\Bigg)^{\frac{1}{s}}
\leqslant \gamma \Lambda(\gamma, 3n, \rho)\, g\left(x_0, \frac{m(\rho)+2(1+b_0)\rho}{\rho}\right),
\end{equation*}
which proves Theorem \ref{thweakHarnck}.

%%%%%%%%%%%%%%%%%%%%%%%%%%%%%%%%%%%%%%%%%%%%%%%%%%%%%%%%%%%%%%%%%%%%%%%%%%%%%%%%%%%%%%%%%%%%%%%%%%%%%%%%%%%
%%%%%%%%%%%%%%%%%%%%%%%%%%%%%%%%%%%%%%%%%%%%%%%%%%%%%%%%%%%%%%%%%%%%%%%%%%%%%%%%%%%%%%%%%%%%%%%%%%%%%%%%%%
%%%%%%%%%%%%%%%%%%%%%%%%%%%%%%%%%%%%%%%%%%%%%%%%%%%%%%%%%%%%%%%%%%%%%%%%%%%%%%%%%%%%%%%%%%%%%%%%%%%%%%%%%%%%%

\section{ Proof of Theorem \ref{sk.th 3.1} (sub-estimate of solutions)}

In this section we prove Theorem \ref{sk.th 3.1}.
%\begin{theorem}\label{sk.th 3.1}
%Let all the assumptions  of Theorem \ref{sk.th 1.1} be fulfilled, then
%\begin{equation}\label{sk 3.1}
%g\left(x_0, \frac{M(\rho)+2(1+b_0)\rho}{\rho}\right)
%\leqslant
%\gamma\, e^{2n\lambda(\rho)}\fint\limits_{B_{5\rho/4}(x_0)}
%g\left(x_0, \frac{u+2(1+b_0)\rho}{\rho}\right) dx.
%\end{equation}
%\end{theorem}
%\begin{proof}

Let's fix $\sigma$, $\sigma_1 \in (0, 1)$,
$\rho\leqslant r< r(1+\sigma\sigma_1)< r(1+\sigma)\leqslant \dfrac{5}{4}\rho$, and
consider a function $\zeta\in C_0^{\infty}\left(B_{r(1+\sigma\sigma_1)}(x_0)\right)$
such that $0\leqslant \zeta \leqslant 1$, $\zeta=1$ in $B_r(x_0)$ and
$|\nabla \zeta|\leqslant (\sigma \sigma_1 r)^{-1}$.
Testing \eqref{gelintidentity} by
$\varphi=\overline{u}\,\mathcal{G}^{s-1}\left(x_0, \overline{u}/\rho\right)\zeta^{\,l}$,
%where $\overline{u}=u+2(1+b_0)\rho$,
$s\geqslant 1$, $l\geqslant \max\{q, s/2\}$,
%$$
%\varphi=\overline{u}\,\mathcal{G}^{s}
%\left(x_0, \overline{u}/\rho\right)\zeta^l,
%\quad \overline{u}=u+2(1+b_0)\rho,
%\ \ s\geqslant 1, \ \ l\geqslant q,
%$$
and using \eqref{gw>pG}, % and the properties of the function $\zeta$,
we have
\begin{multline}\label{Voit3.1}
s\int\limits_{B_{r(1+\sigma\sigma_1)}(x_0)}
G(x, | \nabla u|)\,\mathcal{G}^{s-1}\left(x_0, \overline{u}/\rho\right) \zeta^{\,l}\, dx
\\
\leqslant
l\int\limits_{B_{r(1+\sigma\sigma_1)}(x_0)}
g(x, | \nabla u|)\,\frac{\overline{u}}{\sigma\sigma_{1}\rho\,\zeta}\,
\mathcal{G}^{s-1}\left(x_0, \overline{u}/\rho\right)\zeta^{\,l}\,dx.
\end{multline}
Using \eqref{gYoungineq1} with $\varepsilon=\dfrac{s}{2l}$,
$a=|\nabla u|$, $b=\dfrac{\overline{u}}{\sigma\sigma_{1}\rho\,\zeta}$,
we estimate the right-hand side of \eqref{Voit3.1} from above as follows:
\begin{multline}\label{Voit3.2}
l\int\limits_{B_{r(1+\sigma\sigma_1)}(x_0)}
g(x, | \nabla u|)\,\frac{\overline{u}}{\sigma\sigma_{1}\rho\,\zeta}\,
\mathcal{G}^{s-1}\left(x_0, \overline{u}/\rho\right)\zeta^{\,l}\,dx
\\
\leqslant
\frac{s}{2}\int\limits_{B_{r(1+\sigma\sigma_1)}(x_0)}
G(x, | \nabla u|)\,\mathcal{G}^{s-1}\left(x_0, \overline{u}/\rho\right) \zeta^{\,l}\, dx
\\
+l\int\limits_{B_{r(1+\sigma\sigma_1)}(x_0)}
g\left(x, \frac{\overline{u}}{\varepsilon\sigma\sigma_{1}\rho\,\zeta} \right)
\frac{\overline{u}}{\sigma\sigma_{1}\rho\,\zeta}\,
\mathcal{G}^{s-1}\left(x_0, \overline{u}/\rho\right)\zeta^{\,l}\,dx,
\end{multline}
moreover, since
$\dfrac{\overline{u}}{\varepsilon\sigma\sigma_{1}\rho\,\zeta}
\geqslant \dfrac{\overline{u}}{\rho}
\geqslant 2(1+b_{0})$,
conditions (${\rm g}_{1}$), (${\rm g}_{3}$),  inequality \eqref{G>gw}
and equality $\varepsilon=\dfrac{s}{2l}$ give the estimate
\begin{multline}\label{Voit3.3}
l\int\limits_{B_{r(1+\sigma\sigma_1)}(x_0)}
g\left(x, \frac{\overline{u}}{\varepsilon\sigma\sigma_{1}\rho\,\zeta} \right)
\frac{\overline{u}}{\sigma\sigma_{1}\rho\,\zeta}\,
\mathcal{G}^{s-1}\left(x_0, \overline{u}/\rho\right)\zeta^{\,l}\,dx
\\
\leqslant \frac{c_{1}l}{\varepsilon^{q-1}}\,
\frac{1}{(\sigma\sigma_{1})^{q}}
\int\limits_{B_{r(1+\sigma\sigma_1)}(x_0)}
g\left(x, \frac{\overline{u}}{\rho}\right)\,\frac{\overline{u}}{\rho}\,
\mathcal{G}^{s-1}\left(x_0, \overline{u}/\rho\right)\zeta^{\,l-q}\,dx
\\
\leqslant \frac{\gamma\,l^{q}e^{\lambda(\rho)}}{(\sigma\sigma_{1})^{q}}
\int\limits_{B_{r(1+\sigma\sigma_1)}(x_0)}
\mathcal{G}^{s}\left(x_0, \overline{u}/\rho\right)\zeta^{\,l-q}\,dx.
\end{multline}
Combining \eqref{Voit3.1}, \eqref{Voit3.2}, \eqref{Voit3.3}, we obtain
$$
s\int\limits_{B_{r(1+\sigma\sigma_1)}(x_0)}
G(x, | \nabla u|)\,\mathcal{G}^{s-1}\left(x_0, \overline{u}/\rho\right) \zeta^{\,l}\, dx
\leqslant
\frac{\gamma\,l^{q}e^{\lambda(\rho)}}{(\sigma\sigma_{1})^{q}}
\int\limits_{B_{r(1+\sigma\sigma_1)}(x_0)}
\mathcal{G}^{s}\left(x_0, \overline{u}/\rho\right)\zeta^{\,l-q}\,dx.
$$
In turn, using this inequality, as well as (${\rm g}_{3}$), \eqref{gYoungineq1}
and \eqref{G>gw}, we deduce that
\begin{equation*}\label{sk 3.2}
\begin{aligned}
&\int\limits_{B_{r(1+\sigma\sigma_1)}(x_0)}
\left|\nabla \big[\mathcal{G}^{s}\left(x_0, \overline{u}/\rho\right) \zeta^{\,l}\big]\right| dx
\\
&\leqslant\frac{s}{\rho}\int\limits_{B_{r(1+\sigma\sigma_1)}(x_0)}
\mathcal{G}^{s-1}\left(x_0, \overline{u}/\rho\right)
g\left(x_0, \overline{u}/\rho\right) |\nabla u|\, \zeta^{l}\, dx
+\frac{l}{\sigma\sigma_{1}\rho}\int\limits_{B_{r(1+\sigma\sigma_1)}(x_0)}
\mathcal{G}^{s}\left(x_0, \overline{u}/\rho\right) \zeta^{\,l-1}\, dx
\\
&\leqslant\gamma s\,\frac{e^{\lambda(\rho)}}{\rho}\!\!\!
\int\limits_{B_{r(1+\sigma\sigma_1)}(x_0)}\!\!\!\!\!
G(x, |\nabla u|)\,\mathcal{G}^{s-1}\left(x_0, \overline{u}/\rho\right)
\zeta^{\,l}\, dx
+\gamma s l\,\frac{e^{2\lambda(\rho)}}{\rho}\!\!\!
\int\limits_{B_{r(1+\sigma\sigma_1)}(x_0)}\!\!\!\!\!
\mathcal{G}^s \left(x_0, \overline{u}/\rho\right)  \zeta^{\,l-1}\, dx
\\
&\leqslant \frac{\gamma s l^{q}}{(\sigma\sigma_1)^{q}}
\frac{e^{2\lambda(\rho)}}{\rho}\int\limits_{B_{r(1+\sigma\sigma_1)}(x_0)}
\mathcal{G}^s \left(x_0, \overline{u}/\rho\right)  \zeta^{\,l-q}\, dx.
\end{aligned}
\end{equation*}
Combining this and Sobolev's inequality, %and (\ref{sk 3.2}),
we obtain
\begin{equation}\label{sk 3.3}
\begin{aligned}
\Bigg(&\int\limits_{B_{r}(x_0)}
\mathcal{G}^{\frac{sn}{n-1}} \left(x_0, \overline{u}/\rho\right)dx
\Bigg)^{\frac{n-1}{n}}
\leqslant
\Bigg(\int\limits_{B_{r(1+\sigma\sigma_1)}(x_0)}
\big[\mathcal{G}^{s} \left(x_0, \overline{u}/\rho\right)
\zeta^{\,l}\,\big]^{\frac{n}{n-1}}\, dx\Bigg)^{\frac{n-1}{n}}
\\
&\leqslant
\int\limits_{B_{r(1+\sigma\sigma_1)}(x_0)}
\left|\nabla \big[\mathcal{G}^{s}\left(x_0, \overline{u}/\rho\right) \zeta^{\,l}\,\big]\right| dx
\leqslant\frac{\gamma s\, l^{q}}{(\sigma\sigma_1)^{q}}
\frac{e^{2\lambda(\rho)}}{\rho}\int\limits_{B_{r(1+\sigma\sigma_1)}(x_0)}
\mathcal{G}^s \left(x_0, \overline{u}/\rho\right)dx.
\end{aligned}
\end{equation}

Now, for $i$, $j=0, 1, 2,\ldots$, we define the sequences
%$r_{i,j}:= \dfrac{\rho}{4}\,(5-2^{-i})+\dfrac{\rho}{8}\,2^{-i-j}$
$$
r_{i,j}:= \frac{\rho}{4}\,(5-2^{-i})+\frac{\rho}{8}\,2^{-i-j}, \quad
s_j:= \left( \frac{n}{n-1}\right)^j, \quad
l_j:=q\left( \frac{n}{n-1}\right)^j.
$$
Let $\zeta_{\,i, j}\in C_0^{\infty}\left(B_{r_{i, j}}(x_0)\right)$,
$0\leqslant \zeta_{\,i, j} \leqslant 1$,
$\zeta_{\,i, j}=1$ in $B_{r_{i, j+1}}(x_0)$,
$|\nabla \zeta_{\,i, j}|\leqslant 2^{\,i+j+4}/\rho$.
For $i$, $j=0, 1, 2,\ldots$,
we also set $r_i:= r_{i, \infty}$,
$M_i:= \esssup\limits_{B_{r_i}(x_0)}u$ and
$$
y_{i, j}:= \Bigg(\fint\limits_{B_{r_{i, j}}(x_0)}
\mathcal{G}^{s_j}\left(x_0, \overline{u}/\rho\right)
 dx\Bigg)^{1/s_j}.
$$
%Substituting in \eqref{sk 3.3} $r=r_{i,j+1}$, $r(1+\sigma\sigma_{1})=r_{i,j}$,
%$s=s_{j}$ and $l=l_{j}$,
From \eqref{sk 3.3} we obtain
\begin{equation}\label{sk 3.4}
y_{i, j+1}\leqslant
\Big(\gamma\, 2^{(i+j)\gamma} e^{2\lambda(\rho)} \Big)^{1/s_{j}} y_{i, j},
\ \ i, j= 0, 1, 2,\ldots\,.
\end{equation}
We iterate inequality (\ref{sk 3.4}) with respect to $j$
and use the fact that $r_{i+1}= r_{i, 0}$ to obtain
\begin{multline*}
\mathcal{G} \left(x_0, \frac{M_i+2(1+b_0)\rho}{\rho}\right)
\leqslant\gamma\,2^{i\gamma}e^{2n\lambda(\rho)}
\fint\limits_{B_{r_{i+1}}(x_0)} \mathcal{G} \left(x_0, \overline{u}/\rho\right)dx
\\
\leqslant\gamma\, 2^{i\gamma}e^{2n\lambda(\rho)}\,\frac{M_{i+1}+2(1+b_0)\rho}{\rho}
\fint\limits_{B_{r_{i+1}}(x_0)} g \left(x_0, \overline{u}/\rho\right)dx.
\end{multline*}
This inequality together with \eqref{gYoungineq2} and \eqref{G>gw} implies that, for any
$\varepsilon\in (0, 1)$ and $i=0, 1, 2,\ldots$,
\begin{equation*}\label{sk 3.5}
\begin{aligned}
g& \left(x_0, \frac{M_i+2(1+b_0)\rho}{\rho}\right)
\\[5pt]
&\leqslant\frac{1}{\varepsilon}\,
g \left(x_0, \frac{M_i+2(1+b_0)\rho}{\rho}\right)\frac{M_i+2(1+b_0)\rho}{M_{i+1}+2(1+b_0)\rho}
+ \varepsilon^{p-1} g \left(x_0, \frac{M_{i+1}+2(1+b_0)\rho}{\rho}\right)
 \\[5pt]
&\leqslant \varepsilon^{p-1} g \left(x_0, \frac{M_{i+1}+2(1+b_0)\rho}{\rho}\right)
+\frac{\gamma\,2^{i\gamma}}{\varepsilon}\,
e^{2n\lambda(\rho)}\fint\limits_{B_{5\rho/4}(x_0)}
g \left(x_0, \overline{u}/\rho\right)dx.
\end{aligned}
\end{equation*}
Iterating the resulting inequality, we get for any $i\geqslant 1$
\begin{equation*}
\begin{aligned}
&g \left(x_0, \frac{M(\rho)+2(1+b_0)\rho}{\rho}\right)
= g \left(x_0, \frac{M_0+2(1+b_0)\rho}{\rho}\right)
\\[5pt]
&\leqslant \varepsilon^{i(p-1)} g \left(x_0, \frac{M_i+2(1+b_0)\rho}{\rho}\right)
+\gamma\,  \varepsilon^{-1} e^{2n\lambda(\rho)}
\sum\limits_{j=0}^{i-1}(\varepsilon^{p-1} 2^{\gamma})^j
\fint\limits_{B_{5\rho/4}(x_0)}
g \left(x_0, \overline{u}/\rho\right)dx.
\end{aligned}
\end{equation*}
Finally, choosing $\varepsilon$  from the condition $\varepsilon^{p-1}2^{\gamma}=1/2$
and passing $i$ to infinity, we arrive at
\begin{equation*}
g \left(x_0, \frac{M(\rho)+2(1+b_0)\rho}{\rho}\right) \leqslant
\gamma  e^{2n\lambda(\rho)}\fint\limits_{B_{5\rho/4}(x_0)}
g \left(x_0, \frac{u+2(1+b_0)\rho}{\rho}\right)dx,
\end{equation*}
which completes the proof of Theorem \ref{sk.th 3.1}.
%\end{proof}

%$\mathcal{B}_{1}^{f}$ classes, Kilpel\"{a}inen-Mal\'{y}, $\textbf{W}_{1,p}=...$
%nonstandard growth, Orlicz classes

\vskip3.5mm
{\bf Acknowledgements.} This paper is supported by Ministry of Education and Science of Ukraine
(project numbers are 0118U003138, 0119U100421) and by the Volkswagen Foundation project
"From Modeling and Analysis to Approximation".

\bigskip

CONTACT INFORMATION

\medskip

Maria A.~Shan\\
Vasyl' Stus Donetsk National University,
600-richcha str. 21, 21021 Vinnytsia, Ukraine\\shan$\underline{\phantom{i}}$maria@ukr.net

\medskip
Igor I.~Skrypnik\\Institute of Applied Mathematics and Mechanics,
National Academy of Sciences of Ukraine, Gen. Batiouk str. 19, 84116 Sloviansk, Ukraine
%Vasyl' Stus Donetsk National University,
%600-richcha Str. 21, 21021 Vinnytsia, Ukraine
\\iskrypnik@iamm.donbass.com

\medskip
Mykhailo V.~Voitovych\\Institute of Applied Mathematics and Mechanics,
National Academy of Sciences of Ukraine, Gen. Batiouk str. 19, 84116 Sloviansk, Ukraine\\voitovichmv76@gmail.com

\end{document}